\documentclass[reqno,11pt]{article}
\usepackage{a4wide}
\usepackage{color,eucal,enumerate,mathrsfs}
\usepackage[normalem]{ulem}
\usepackage{amsmath,amssymb,epsfig,bbm}
\numberwithin{equation}{section}

\usepackage[latin1]{inputenc}

\usepackage[pdfborder={0 0 0}]{hyperref}

\newenvironment{proof}{\removelastskip\par\medskip   
\noindent{\em proof} \rm}{\penalty-20\null\hfill$\square$\par\medbreak}

\DeclareMathOperator*{\esssup}{ess\,sup}
\DeclareMathOperator*{\essinf}{ess\,inf}

\newtheorem{theorem}{Theorem}[section]

\newtheorem{proposition}[theorem]{Proposition}

\newtheorem{definition}[theorem]{Definition}

\newtheorem{remark}[theorem]{Remark}

\newcommand{\beq}{\begin{equation}}
\newcommand{\eeq}{\end{equation}}

\newcommand{\lleq}{\prec}
\newcommand{\ggeq}{\succ}
\newcommand{\bd}{{\mathbf\Delta}}

\newcommand{\CD}{{\sf CD}}
\newcommand{\RCD}{{\sf RCD}}

\newcommand{\R}{\mathbb{R}}

\newcommand{\mm}{{\mbox{\boldmath$m$}}}

\newcommand{\ggamma}{{\mbox{\boldmath$\gamma$}}}

\newcommand{\sfd}{{\sf d}}

\newcommand{\Kliminf}{K\kern-3pt-\kern-2pt\mathop{\rm lim\,inf}\limits}  
\newcommand{\supp}{\mathop{\rm supp}\nolimits}   
\newcommand{\Lip}{\mathop{\rm Lip}\nolimits}          
\renewcommand{\d}{{\mathrm d}}

\newcommand{\restr}[1]{\lower3pt\hbox{$|_{#1}$}}
\newcommand{\la}{{\langle}}                  
\newcommand{\ra}{{\rangle}}
\newcommand{\eps}{\varepsilon}  
\newcommand{\nchi}{{\raise.3ex\hbox{$\chi$}}}

\newcommand{\lims}{\varlimsup}

\newcommand{\fr}{\hfill$\blacksquare$}                      

\newcommand{\probt}[1]{\mathscr P_2(#1)}                   

\renewcommand{\mm}{\mathfrak m}                                

\title{The abstract Lewy-Stampacchia inequality and applications}
\author{Nicola Gigli\thanks{Institut de math\'ematiques de Jussieu - UPMC}, Sunra Mosconi\thanks{Universit\`a di Catania}}

\begin{document}
\maketitle
\begin{abstract}
We prove an abstract and general version of the Lewy-Stampacchia inequality. We then show how to reproduce more classical versions of it and, more importantly, how it can be used in conjunction with Laplacian comparison estimates to produce large class of functions with bounded Laplacian on spaces with a lower bound on the Ricci curvature. 
\end{abstract}
\tableofcontents
\section{Introduction}
The Lewy-Stampacchia inequality \cite{LewyStampacchia70} is a classical inequality concerning the solution of the obstacle problem. It can be stated as follows: let $\Omega\subset \R^d$ be a given open bounded set, $\varphi\in C^\infty_c(\Omega)$ and $u$ the minimum of $E(v):=\frac12\int|\nabla v|^2\,\d\mathcal L^d$ among all $v\in W^{1,2}_0(\Omega)$ with $v\geq \varphi$. Then
\begin{equation}
\label{eq:lsintro}
0\land\Delta\varphi \leq\Delta u\leq 0.
\end{equation}
Here the inequality $\Delta u\leq 0$ is obvious because $u$ minimizes the energy among positive perturbations. To see why $0\land \Delta\varphi\leq\Delta u$ holds, very informally, notice that where $\{u>\varphi\}$, $u$ is harmonic and hence $\Delta u=0$, while where $\{u=\varphi\}$ we have  $\Delta u=\Delta\varphi$ (the precise derivation of \eqref{eq:lsintro} - which we do not discuss in this introduction - must take into account what happens   at the boundary of the set $\{u=\varphi\}$).

Over time, inequality \eqref{eq:lsintro} has been  generalized in several different directions, among others we mention \cite{Troianiello87} as a general reference for general linear operators and boundary values, \cite{MokraneMurat04} for nonlinear Leray-Lions operators, \cite{Rodrigues-Teymurazyan11} for nonlinear $p$-Laplacian type operators, and \cite{Servadei-Valdinoci13} for the fractional Laplacian and the Laplacian in the Heisenberg group (see also \cite{PinamontiValdinoci13} for this latter setting).  

The first scope of this paper is to further investigate the structure of the Lewy-Stampacchia inequality and to provide an abstract version of it in the context of topological vector lattices $(B,\tau,\lleq)$, see Theorem \ref{LS} for the precise formulation.  Beside the higher level of abstraction we reach, our approach is new in the sense that it does not rely on differentiability properties of the convex functional $E:B\to\R\cup\{+\infty\}$ considered, nor on the strict $\cal T$-monotonicity of its subdifferential, the latter meaning that
\begin{subequations}
\begin{align}
\label{eq:tmonotintro}
\langle u^*-v^*, (u-v)\lor 0\rangle&\geq 0 \qquad \phantom{\Leftrightarrow}& &\forall u^*\in \partial E(u), \ v^*\in \partial E(v),\\
\label{eq:tmonotintro2}
\langle u^*-v^*, (u-v)\lor 0\rangle&= 0 \qquad \qquad\Leftrightarrow & &u\lleq v.
\end{align}
\end{subequations}
In fact, what turns out to be crucial is the submodularity property, also called boolean subadditivity, of the functional itself, i.e.:
\begin{equation}
\label{eq:subintro}
E(u\land v)+E(u\lor v)\leq E(u)+E(v),\qquad\forall u,v\in B.
\end{equation}
While submodular functions are well established tools in discrete optimization, their r\^ole in the general theory of vector lattices, and in Lewy-Stampacchia type estimates in particular, has not yet, as far as we  are aware of, been recognized. 

Being a zeroth-order condition, verifying the submodularity \eqref{eq:subintro} for a given functional is a much easier and a more direct task than obtaining the strict $\cal T$-monotonicity of its differential, which, especially for non-differentiable functionals, requires a good knowledge of the subdifferential itself. Moreover, submodularity is a weaker condition, since at the derivative level is equivalent to $\cal T$-monotonicity \eqref{eq:tmonotintro} of the differential instead of its {\em strict} $\cal T$-monotonicity \eqref{eq:tmonotintro},\eqref{eq:tmonotintro2}.

\bigskip

Beside recovering the classical Lewy-Stampacchia inequality for the Laplacian and showing how to quickly re-obtain the one for the fractional Laplacian (recently proved in \cite{Servadei-Valdinoci13}), we apply the abstract formulation to the double obstacle problem on $\CD^*(K,N)$ spaces, which was in fact  the main motivation for starting this project. $\CD^*(K,N)/ \RCD^*(K,N)$ spaces  are metric measure structures which, in a sense, resemble Finslerian/Riemannian manifolds with Ricci curvature bounded from below by $K$ and dimension bounded from above by $N$,  see \cite{Lott-Villani09}, \cite{Sturm06I}, \cite{Sturm06II}, \cite{AmbrosioGigliSavare11}, \cite{Gigli12}, \cite{Erbar-Kuwada-Sturm13} for the relevant definitions.

Being the curvature-dimension condition a second-order notion, one expects the presence of `many' functions with some sort of second order regularity. Yet, priori to the present manuscript the only smoothing tool available was regularization with the heat flow which, due to fast diffusion, offers little control on the local behavior of the regularized functions.

Here we couple the Lewy-Stampacchia inequality with the Laplacian comparison estimates proved in \cite{Gigli12} to produce `constrained' functions with bounded Laplacian. In particular, on $\CD^*(K,N)$ spaces we shall build cut-off functions and regularized Kantorovich potentials from intermediate times along a geodesic, in both cases producing functions with bounded Laplacian. See Theorems \ref{thm:cutoff} and \ref{thm:regkant} for the precise formulation.  The relevance of having smooth cut-off functions is clear, on the other hand having smooth Kantorovich potentials seems crucial in order to be able to differentiate functionals along a $W_2$-geodesic, see for instance the discussion at the end of \cite{Gigli13}. 

We remark that cut-off functions were already built in \cite{AMS13} on $\RCD(K,\infty)$ spaces, but the technique seems not applicable to the class of $\CD^*(K,N)$ spaces. 

In the stricter $\RCD^*(K,N)$ class  our construction produces Lipschitz functions. This regularity result has little to do with the Lewy-Stampacchia inequality, but is rather based on Lipschitz continuity  of harmonic functions on $\RCD^*(K,N)$ spaces recently obtained in \cite{Jiang13} (see also \cite{Kell13} and \cite{Hua-Kell-Xia13}) together with quite standard techniques in the setting of the obstacle problem, see Section \ref{se:lip}.

\section{The abstract Lewy-Stampacchia inequality}
\subsection{Topological vector lattices}
Here we briefly introduce the basic notions needed to state the Lewy-Stampacchia inequality in an abstract framework, referring to \cite{Peressini67} for a detailed discussion about ordered topological vector spaces.

A \emph{lattice} $(S, \lleq )$ is given by a set $S$ with a partial ordering $\lleq$ such that for every $x, y\in S$ there exist elements $x\lor y,x\land y\in S$ satisfying
\[
\begin{array}{rl}
\qquad x \lleq  x\lor y,\\
y \lleq x\lor y,\\
x\lleq z, \ y\lleq z\quad \Rightarrow \quad &x\lor y\lleq z,
\end{array}
\qquad\textrm{and}\qquad
\begin{array}{rl}
  x\land y\lleq x,\\
x\land y\lleq y,\\
z\lleq x,\ z\lleq y\quad \Rightarrow \quad &z\lleq x\land y.
\end{array}
\]
Given two elements $x,y$ of a lattice $S$ with $x \lleq y$ we denote by $[x,y]\subset S$ the interval defined by  $x$ and $y$, i.e.:
\[
[x,y]:=\{z\in S\ :\ x \lleq z \lleq y\}.
\]
Similarly, by $]-\infty,x]$ we intend the set $\{z: z\lleq x\}$ and by $[x,+\infty[$ the set  $\{z: x\lleq z\}$. Subsets of $S$ contained in some interval $[x,y]$ are called order-bounded.

\begin{definition}[Topological vector lattice]
A \emph{topological vector lattice} $(B,\tau,\lleq)$ is a Hausdorff locally convex topological vector space $(B,\tau)$ endowed with a lattice structure compatible with the vector one in the sense that for given $x,y\in B$ with $x\lleq y$ we have
\[
\begin{split}
 x+z&\lleq y+z,\qquad\forall z\in B,\\
 \lambda x&\lleq\lambda  y,\qquad\forall \lambda\in \R,\ \lambda\geq 0.
\end{split}
\]
The \emph{positive cone} $P\subseteq B$ is the convex cone $\{x\in B: 0\lleq x\}$, and $x\lleq y$ iff $x-y\in P$.
\end{definition}

At this level of generality, there is no connection between the topology and the lattice structure on $B$. Notice that this may be in contrast with some terminology, where a topological vector lattice usually requires that there is a neighborhood basis for $0$ consisting of solid sets. 

Let us recall that $P\cap -P=\{0\}$ (i.e. $P$ is proper) due to the antisymmetry of $\lleq$ and $P-P=B$ (i.e. $P$ is generating) due to $(B, \lleq)$ being a lattice.
The \emph{order dual} of $B$ is denoted by $B'_\lleq$ and consists of all the real valued linear functionals on $B$ which are bounded on order-bounded sets. Being $B$ a lattice, $B'_\lleq$ is a vector lattice w.r.t. to the ordering induced by the dual cone $P'=\{l\in B'_\lleq: l(x)\geq 0 \ \forall x\in P\}$. In particular, for any $l, m\in B'_\lleq$, the Riesz-Kantorovich formulae hold for any $x\in P$:
\begin{equation}
\label{RKformula}
l\lor m (x):= \sup_{z\in[0,x]}\{l(z)+m(x-z)\},\qquad\qquad l\land m(x):=\inf_{z\in[0,x]}\{l(z)+m(x-z)\}.
\end{equation}
We shall denote by $B^*$ the topological dual of $(B, \tau)$ and by $\langle \cdot, \cdot\rangle:B^*\times B\to \R$ the corresponding duality pairing.
The topological dual convex cone $P^*\subset B^*$ of $P$ is
\[
P^*:=\{x^*\in B^*\ :\ \la x^*,x\ra\geq 0,\ \forall x\in P\},
\]
and we will still denote by  $\lleq$ the partial order structure induced by $P^*$ on $B^*$. We then define the \emph{topological lattice dual} as
\[
B^*_\lleq:= P^*-P^*.
\]
Any $x^*\in P^*$ is bounded on order bounded sets, so that in general it holds
\begin{equation}
\label{Bprec}
B^*_\lleq\subseteq B^*\cap B'_\lleq.
\end{equation}
It is obvious by definition that $P^*$ is weakly$^*$-closed in $B^*$. On the other hand, if $P$ is closed then $B^*_\lleq$ is weakly$^*$-dense. Indeed, if $x\in B$ is such that $\la x^*, x\ra=0$ for any $x^*\in B^*_\lleq$, then in particular $\la x^*, x\ra\geq 0$ for every $x^*\in P^* $, which by the bipolar theorem  gives $x\in P$, and $\la x^*, x\ra\geq 0$ for any $x^*\in -P^*$, so that $x\in -P$ and hence $x\in P\cap-P=\{0\}$. 

It turns out that $B^*_\lleq$ is also a vector lattice. To check this, it suffices to show that the Riesz-Kantorovich formulae \eqref{RKformula} provide continuous linear functionals. If $x^*=x_1^*-x_2^*$ and $y^*=y_1^*-y_2^*$ with $x_i^*$, $y_i^*\in P^*$ for $i=1,2$, then for any $0\lleq z\lleq x$ it holds
\[
\la x^*, z\ra+\la y^*, x-z\ra\leq \la x_1^*, z\ra+\la y_1^*, x-z\ra\leq  \la x_1^*+y_1^*, x\ra
\]
and similarly
\[
\la x^*, z\ra+\la y^*, x-z\ra\geq -\la x_2^*+y_2^*, x\ra.
\]
Therefore both $x^*\lor y^*$ and $x^*\land y^*$ are (topologically) bounded linear functionals on $P$, which have unique continuous extension to the whole $B$ due to $P$ being generating.

Although unnecessary in our discussion, we remark that if $P'\subseteq B^*$, then $B'_\lleq\subseteq B^*$ and thus equality holds in \eqref{Bprec}. By Proposition 2.16, Chapter 2 in \cite{Peressini67} this is the case, for example, if $(B, \tau)$ is a complete metrizable t.v.s. of second category and $P$ is closed, as in the applications we will propose.

\bigskip

Given a convex function $E:B\to\R\cup\{+\infty\}$ we shall denote by ${\rm dom}\,E\subset B$ the set  $\{x:E(x)<\infty\}$. For  $x_0\in B$, the subdifferential $\partial E(x_0)\subset B^*$ of $E$ at $x_0$ is defined to be the empty set if $x_0\notin{\rm dom}\, E$ and otherwise as the (possibly empty) set of elements $x^*\in B^*$  such that
\[
\langle x^*, x-x_0\rangle\leq E(x)-E(x_0)\quad \forall x\in B.
\]
The \emph{domain} of $\partial E$ is defined as ${\rm dom}(\partial E)=\{x\in B:\partial E(x)\neq \emptyset\}$. 

Given a lattice $(S,\lleq)$, a function $E:(S, \lleq )\to \R\cup\{+\infty\}$ is said to be {\em submodular} provided
\beq
\label{order}
E(x\land y)+E(x\lor y)\leq E(x)+E(y), \qquad \qquad \forall x, y\in S.
\eeq
We shall be interested in topological vector lattices and functionals $E$ which are both convex and submodular. Notice that for such $E$'s the subdifferential satisfies the following variant of the classical monotonicity property, known as $\cal T$-monotonicity:
\beq
\label{monot}
\langle x^*-y^*, (x-y)\lor 0\rangle\geq 0 \qquad \forall x^*\in \partial E(x), \ y^*\in \partial E(y),
\eeq
(one says that $\partial E $ is strictly $\cal T$-monotone provided equality in \eqref{monot} implies $x\lleq y$).
Indeed, by definition of subdifferential we have
\[
\begin{split}
\langle x^*, (y-x)\land 0\rangle\leq E(x+(y-x)\land 0)-E(x)=E(x\land y)-E(x),\\
\langle y^*, (x-y)\lor  0\rangle\leq E(y+(x-y)\lor 0)-E(y)=E(x\lor y)-E(y),
\end{split}
\]
so that adding up the inequalities and noticing that
\[
\langle x^*, (y-x)\land 0\rangle+\langle y^*, (x-y)\lor  0\rangle=\langle x^*-y^*, (y-x)\land 0\rangle
\]
we get the claim. The same argument shows that for convex $E$'s with ${\rm dom}(\partial E)=B$, the $\cal T$-monotonicity property \eqref{monot} yields the submodularity \eqref{order}.

\begin{remark}\label{submodsmooth}{\rm It might be useful to recall that if $f$ is a smooth function defined on $\R^d$ and the latter is endowed with its natural lattice structure given by $x\lleq y$ if all the components of $x$ are $\leq$ than the corresponding ones of $y$, then $f$ is submodular if and only if
\[
\frac{\partial^2f}{\partial x_i\partial x_j}(x)\leq 0,\qquad\forall x\in \R^d,\ i,j=1,\ldots,d,\ i\neq j.
\]

}\fr\end{remark}
We conclude this section recalling a basic result in convex analysis we shall use in our proof of the Lewy-Stampacchia inequality, see Theorem 2.9.1 in  \cite{Zalinescu} for a proof.
\begin{theorem}\label{convex}
Let $B$ be a Hausdorff locally convex topological vector space, $C\subset B$ a convex set and $E:B\to\R\cup\{+\infty\}$ a convex function. Assume that either ${\rm dom}\, E\cap {\rm int}\, C\neq\emptyset$ or that there exists $x\in {\rm dom}\, E\cap C$ where $E$ is continuous.

Then $\bar x\in C$ realizes the minimum of $E$ in $C$ if and only if there exists $x^*\in\partial E(\bar x)$ such that
\[
\la x^*,x-x_0\ra\geq 0,\qquad\forall x\in C.
\]
\end{theorem}

\subsection{Abstract formulation of the Lewy-Stampacchia inequality}

We can now prove a general version of the Lewy--Stampacchia inequality.

\begin{theorem}[Abstract  Lewy-Stampacchia inequality]
\label{LS} Let $(B,\tau,\lleq)$ be a topological vector lattice, and  $E:B\to \R\cup\{+\infty\}$ a convex and submodular functional. Furthermore, let  $\varphi, \psi\in B$ with $\varphi\lleq \psi$ and $\bar u\in B$  a minimizer for $E$ on $[\varphi, \psi]$. 

Assume that  either  ${\rm dom E}\cap{\rm int}\,[\varphi, \psi]\neq\emptyset$ or   that there exists $u\in {\rm dom}\, E\cap  [\varphi,\psi]$ where $E$ is continuous.

Then
\beq
\label{LSin}
\begin{split}
 &\forall w_1^*\in \partial E(\varphi)\cap B^*_\lleq\ \  \exists x_1^*\in \partial E(\bar u)\cap B^*_\lleq\quad\text{ such that }\quad x_1^*\lleq w_1^*\lor 0, \\
&\forall  w_2^*\in \partial E(\psi)\cap B^*_\lleq\ \ \exists x_2^*\in \partial E(\bar u)\cap B^*_\lleq\quad\text{ such that }\quad w_2^*\land 0\lleq x_2^*.
\end{split}
\eeq
\end{theorem}
\begin{proof} We start proving the first assertion in \eqref{LSin}. Without loss of generality we assume that $\partial E(\varphi)\cap B^*_\lleq\neq\emptyset$ (thus in particular $E(\varphi)<+\infty$) and pick  $w_1^*\in \partial E(\varphi)\cap B^*_\lleq$.  Consider the convex functional $A_1:B\to\R\cup\{+\infty\}$ defined by
\[
A_1(u):=E(u)-\langle w_1^*\lor 0, u\rangle, \qquad \forall u\in B.
\]
We claim that
\beq
\label{LSclaim1}
\inf_{]-\infty, \bar u]} A_1=\inf_{[\varphi, \bar u]} A_1.
\eeq
The inequality $\leq$ is obvious. To   prove the other one,  it suffices to prove that for any $u\lleq \bar u$ it holds $A_1(u) \geq    A_1(u\lor \varphi)$. Suppose not: then for some $u\lleq \bar u$ it holds
\begin{equation}
\label{eq:tardi}
E(u\lor \varphi)-\langle w_1^*\lor 0, u\lor\varphi\rangle>E(u)-\langle w_1^*\lor 0, u\rangle.
\end{equation}
In particular $E(u)<+\infty$ and using  \eqref{order} we get $E(u\lor\varphi), E(u\land\varphi)<+\infty$. Moreover
\[
E(\varphi)-E(u\land \varphi)\geq E(u\lor \varphi)-E(u)\stackrel{\eqref{eq:tardi}}>\langle w_1^*\lor 0, u\lor \varphi-u\rangle\geq \langle w_1^*, u\lor \varphi-u\rangle,
\]
where in the last inequality we used the fact that $ u\lor\varphi-u\ggeq 0$. Recalling that $u-u\lor\varphi=u\land\varphi-\varphi$ we deduce
\[
E(u\land\varphi)<E(\varphi)+\langle w_1^*, u\land\varphi-\varphi\rangle,
\]
which contradicts $w_1^*\in \partial E(\varphi)$. Thus \eqref{LSclaim1} is proved. Now we claim that
\beq
\label{LSclaim2}
\inf_{[\varphi, \bar u]} A_1=A_1(\bar u),
\eeq
and again we argue by contradiction. Hence suppose that for some $u\in [\varphi, \bar u]$ it holds
\[
E(u)-\langle w_1^*\lor 0, u\rangle<E(\bar u)-\langle w_1^*\lor 0, \bar u\rangle.
\]
Then, being $u\lleq \bar u$ we get
\[
E(\bar u )-E(u)>\langle w_1^*\lor 0, \bar u-u\rangle\geq 0, 
\]
which, since $u\in [\varphi, \bar u]\subseteq [\varphi, \psi]$, contradicts the minimality of $\bar u$ in $[\varphi, \psi]$. 

From \eqref{LSclaim1} and \eqref{LSclaim2} we deduce that  $\bar u$ is a minimum for $A_1$ on the convex set $]-\infty, \bar u]$, therefore by Theorem \ref{convex} we deduce the existence of  $y_1^*\in \partial A_1(\bar u)$ such that
\beq
\label{varin}
\langle y_1^*, u-\bar u\rangle \geq 0, \qquad \forall u\in\  ]-\infty, \bar u].
\eeq
Since $y_1^*\in \partial A_1(\bar u)=\partial E(\bar u)-w_1^*\lor 0$,
there exists $x_1^*\in \partial E(\bar u)$ such that $y_1^*=x_1^*-w_1^*\lor 0$.
Letting $u=\bar u-v$ for arbitrary $v\in P$ in \eqref{varin}, we get
\[
\langle x_1^*-w_1^*\lor 0, v\rangle\leq 0, \qquad \forall v\in P,
\]
proving that $x_1^*- w_1\lor 0^*\in-P^*$, which is the first inequality in \eqref{LSin}.

To prove the other one we consider, for any $w_2^*\in \partial E(\psi)$, the functional
\[
A_2(u): =E(u)-\langle w_2^*\land 0, u\rangle,
\]
and arguing as before we prove that  $\bar u$ minimizes $A_2$ over $[\bar u, +\infty[$, thus getting the conclusion along the same lines just used.
\end{proof}

\section{Applications}

\subsection{Recovering the classical case}
Here we show how the general Theorem \ref{LS} yields the classical formulation of the Lewy-Stampacchia inequality. Let $\Omega\subset \R^d$ be an open set and observe that the Hilbert space $W^{1,2}_0(\Omega)$ endowed with the standard ordering given by pointwise a.e.\ inequality
\[
u\lleq v\qquad\stackrel{\rm def}\Longleftrightarrow\qquad\ u(x)\leq v(x),\quad a.e.\ x\in\Omega,
\] 
is a topological vector lattice. Its topological dual is denoted by $W^{-1,2}(\Omega)$ and its topological lattice dual by $W^{-1,2}_\lleq(\Omega)\subset W^{-1,2}(\Omega)$.

The distributional Laplacian $\Delta u$ of a function $u\in W^{1,2}_0(\Omega)$ acts on smooth compactly supported functions as 
\[
C^\infty_c(\Omega)\ni \eta\qquad\mapsto\qquad\la\Delta u,\eta\ra:= \int_\Omega \Delta\eta \,u\,\d\mathcal L^d,
\]
and since $u\in W^{1,2}_0(\Omega)$ we have 
\[
\left| \int_\Omega \Delta\eta \,u\,\d\mathcal L^d\right|= \left|\int_\Omega \nabla\eta \cdot\nabla u\,\d\mathcal L^d\right|\leq \|\nabla \eta\|_{L^2} \|\nabla u\|_{L^2}\leq\| \eta\|_{W^{1,2}_0(\Omega)} \| u\|_{W^{1,2}_0(\Omega)},
\]
which shows that the distributional Laplacian uniquely extends to the element in  $W^{-1,2}(\Omega)$, still denoted by $\Delta u$, given by
\[
W^{1,2}_0(\Omega)\ni \eta\qquad\mapsto\qquad\la\Delta u,\eta\ra:=- \int_\Omega \nabla\eta \cdot\nabla u\,\d\mathcal L^d.
\]
In this sense, it has a meaning to ask whether $\Delta u\in W^{-1,2}_\lleq(\Omega)$ for some $u\in W^{1,2}_0(\Omega)$. The Lewy-Stampacchia inequality can then  be stated as follows:
\begin{theorem}[Classical Lewy-Stampacchia inequality]\label{LSclassic}
Let $\Omega\subset\R^d$ be a bounded open set and $\varphi,\psi\in W^{1,2}_0(\Omega)$ with $\varphi\leq\psi$ a.e.\ and such that  $\Delta\varphi,\Delta\psi\in W^{-1,2}_\lleq(\Omega)$.

Let $\bar u$ be the minimum of $u\mapsto\int_\Omega|\nabla u|^2\,\d\mathcal L^d$ among all functions $u\in W^{1,2}_0(\Omega)$ such that $\varphi\leq u\leq \psi$ a.e. (whose existence and uniqueness follows by standard means in calculus of variations). Then we have  $\Delta \bar u\in W^{-1,2}_\lleq (\Omega)$  as well with 
\begin{equation}
\label{eq:LSbase}
\Delta\varphi\land 0\lleq\Delta \bar u\lleq \Delta\psi\lor 0.
\end{equation}
\end{theorem}
\begin{proof}
The functional $E:W^{1,2}_0(\Omega)\to[0,\infty)$ given by $E(u):=\frac12\int_\Omega|\nabla u|^2$ is clearly convex and  continuous. Moreover, $E$ is submodular (actually, with equality holding in \eqref{order} for every $u,v\in W^{1,2}_0(\Omega)$) as a consequence of the locality property of the gradient:
\[
\nabla u=\nabla v,\qquad \mathcal L^d\text{-a.e.\ on}\ \{u=v\}.
\]
Now observe that the subdifferential of $E$ at $u\in W^{1,2}_0(\Omega)$ is nothing but $-\Delta u\in W^{-1,2}(\Omega)$. Indeed, the trivial inequality
\[
E(u)-\la\Delta u,\eta\ra \leq E(u+\eta),\qquad\forall u,\eta\in W^{1,2}_0(\Omega),
\]
shows that $-\Delta u\in \partial E(u)$ and conversely  testing the subdifferential inequality with $\eps\eta$ for arbitrary $\eta\in W^{1,2}_0(\Omega)$ and then letting $\eps\to 0$ we see that  $-\Delta u$ is the only element in $\partial E(u)$.

The conclusion then comes applying Theorem \ref{LS}.
\end{proof}
Some comments:
\begin{itemize}
\item[i)] The assumption that the obstacles $\varphi,\psi$ have 0 boundary data has been made to simplify the exposition but is in fact unnecessary, see Remark \ref{rem:dir} for some details on how to remove it.

\item[ii)] We stated the thesis in \eqref{eq:LSbase} as an inequality between linear functionals on $W^{-1,2}(\Omega)$. Equivalently, one can interpret it as inequality between measures, due to the fact that elements of the space $W^{-1, 2}_\lleq(\Omega)$  can be faithfully represented as Radon measures. This can be achieved either calling into play the notions of capacity, polar sets and representatives quasi-everywhere defined of  Sobolev functions (see e.g. Chapter 3 of \cite{MaRockner92}), or along the following lines. 

Consider a positive functional $L\in W^{-1, 2}_\lleq(\Omega)$. By restriction it  defines a positive linear functional on ${\rm Lip}_c(\Omega)\subseteq W^{1,2}(\Omega)$ and since for every non-negative $f\in C_c(\Omega)$ there exists  $g\in {\rm Lip}_c(\Omega)$ such that $f(x)\leq g(x)$ for every $x\in\Omega$, such positive linear functional can be uniquely extended to a positive linear functional on $C_c(\Omega)$ (see also the general construction in   Corollary 2.8, Chapter 2 in \cite{Peressini67}). By the Riesz represetation theorem  we get that there exists a non-negative Radon measure $\mu_L$ on $\Omega$ such that 
\[
L(u)=\int_\Omega u\, \d\mu_L \qquad \forall u\in {\rm Lip}_c(\Omega),
\]
and such $\mu_L$ is \emph{unique} by the density of ${\rm Lip}_c(\Omega)$ in $C_c(\Omega)$.  Clearly then, there is a well defined (linear) map $W^{-1,2}_\lleq(\Omega)\ni L\mapsto \mu_L\in {\cal M}(\Omega)$ where we denoted with ${\cal M}(\Omega)$ the set of Radon measures on $\Omega$. We say that this representation is faithful in the sense that the map $L\mapsto \mu_L$ is \emph{injective}, being ${\rm Lip}_c(\Omega)$ strongly dense in $W^{1,2}_0(\Omega)$.
 
Due to this discussion, we will sometime shortly say that the elements of $W^{-1,2}_\lleq(\Omega)$``are" measures.
\item[iii)] Although the  Lewy-Stampacchia inequality can be certainly stated for smooth obstacles, in fact it is more natural - and evidently more general - to formulate it as in the statement we gave, i.e.\ for obstacles having measure valued distributional Laplacian, the latter being intended as in point (ii) above.  It is for this reason that the topological vector lattice considered has been $W^{1,2}_0(\Omega)$ rather than $L^2(\Omega)$. Indeed, convex functionals in $L^2$ have subdifferential which, by definition, must act continuously on $L^2$ functions, which is certainly not the case for a generic measure-valued distributional Laplacian of a Sobolev function. 

In the present case, the version with measure-valued Laplacian could in fact be obtained from the case of smooth obstacles with a quite standard approximation/convergence argument, so that this distinction might be not so relevant. It becomes instead crucial on metric measure spaces, where approximation procedures are not easily available, and in fact the study of the double obstacle problem has as primary goal the one of building `smooth' functions.

\end{itemize}

\subsection{The fractional Laplacian}
We now show how to deduce from Theorem \ref{LS} the Lewy-Stampacchia inequality for the fractional Laplacian, thus reproducing a result already appeared in \cite{Servadei-Valdinoci13} with a simplified argument.

Let  $\Omega\subset \R^d$ be  an open subset, $s\in(0,1)$ and  the space $W^{s,2}_0(\Omega)$ be defined as the closure of $C^\infty_c(\Omega)$ w.r.t.\  the norm
\[
\| u\|_{W^{s,2}_0(\Omega)}^2:=\|u\|_{L^2}^2+\int_{\R^d\times \R^d} \frac{|u(x)-u(y)|^2}{|x-y|^{d+2s}}\,\d x\, \d y.
\]
Clearly, $W^{s,2}_0(\Omega)$ is a lattice w.r.t.\ the a.e.\ ordering  and a Hilbert space with the latter norm, with dual denoted by $W^{-s,2}(\Omega)$ and order dual by $W^{-s,2}_\lleq(\Omega)$. As a general reference for this space and related ones see Chapter 1, Section 5 of \cite{Triebel01}. Notice that, with the same arguments of the previous section, one can see that functionals in $W^{-s,2}_\lleq(\Omega)$  can be faithfully represented as integral w.r.t.\ appropriate Radon measures.

If $\Omega$ is bounded, the functional $E:W^{s,2}_0(\Omega)\to \R$ given by
\[
E(u):=\frac 1 2\int_{\R^d\times\R^d} \frac{|u(x)-u(y)|^2}{|x-y|^{d+2s}}\,\d x\, \d y,
\]
is convex, continuous and coercive (see \cite{Servadei-Valdinoci13} for this latter property). Its subdifferential is related to the fractional Laplacian via the identity
\[
\partial E(u)=(-\Delta)^s u,\qquad\forall u\in W^{s,2}_0(\Omega),
\]
we refer to \cite{DNPV} for the definition and basic properties of the fractional Laplacian.

We claim that $E$ is submodular. To prove this it is sufficient to show that
\[
(x_1-x_2)^2+(y_1-y_2)^2\geq (x_1\lor y_1-x_2\lor y_2)^2+(x_1\land y_1-x_2\land y_2)^2,
\]
for any  $x_1, x_2, y_1, y_2\in \R$. More generally, we shall prove that for every $f:\R\to\R$ convex we have
\begin{equation}
\label{eq:sub2}
f(x_1-x_2)+f(y_1-y_2)\geq f(x_1\lor y_1-x_2\lor y_2)+f(x_1\land y_1-x_2\land y_2),
\end{equation}
for any $ x_1, x_2, y_1, y_2\in \R$. To this aim, let  $g:\R^2\to\R$ given by $g(x_1,x_2)=f(x_1-x_2)$ and endow  $\R^2$  with its natural lattice structure given by component-wise ordering. If $f$ is smooth, then   the identity $\frac{\d^2}{\d x_1\d x_2}g(x_1,x_2)=-\frac{\d^2}{\d x^2}f(x_1-x_2)\leq 0$ and Remark \ref{submodsmooth} show that $g$ is submodular, which is equivalent to the validity of  \eqref{eq:sub2} for any $ x_1, x_2, y_1, y_2\in \R$. The general case follows by approximation.

Collecting together these observations and using Theorem \ref{LS} we deduce:
\begin{theorem}[Lewy-Stampacchia inequality for the fractional Laplacian]
Let $\Omega\subseteq \R^d$ be open $\varphi, \psi\in W^{s, 2}_0(\Omega)$ and $\bar u$ a minimizer for $E$ over $[\varphi, \psi]\neq \emptyset$. Assume that  $(-\Delta)^s\varphi, (-\Delta)^s\psi\in W^{-s,2}_\lleq(\Omega)$. Then $ (-\Delta)^s \bar u\in W^{-s,2}_\lleq(\Omega)$ with
\[
(-\Delta)^s\psi\land 0\lleq (-\Delta)^s \bar u\lleq (-\Delta)^s \varphi\lor 0.
\]
\end{theorem}

\begin{remark}{\rm
An analogous statement holds for arbitrary summability index $p>1$ on the derivative, thus giving a Lewy--Stampacchia inequality for the fractional $p$-Laplacian: inequality \eqref{eq:sub2} with $f(x):=|x|^p$ grants the submodularity of the corresponding functional
\[
E(u):=\frac 1p\int_{\R^d\times\R^d} \frac{|u(x)-u(y)|^p}{|x-y|^{d+ps}}\,\d x\, \d y.
\]
}\fr\end{remark}

\begin{remark}\label{rem:dir}{\rm
The assumption that  $\varphi, \psi$   in Theorem \ref{LS} have 0 boundary data has been made only to simplify the statement. For classical obstacle problems with Dirichlet boundary conditions the naturally available obstacles need not vanish on the boundary of the domain, and may fail to belong, in our abstract setting, to the minimization space $B$. Nevetheless, a minimizer for $E$ over $[\phi, \psi]$ still satisfies a form the Lewy--Stampacchia inequality. Suppose that $E$ is naturally defined on a bigger topological vector lattice $\tilde B$ with continuous, order preserving embedding $B\hookrightarrow \tilde B$, and $\varphi, \psi\in \tilde B$. Consider the subdifferential of $E$ w.r.t.\  to $B$ defined as
\[
\partial_BE(u):=\partial G_u(0)\subseteq B^*, \qquad B\ni v\mapsto G_u(v):=E(v+u).
\]
Using this notion, Theorem \ref{LS} holds with obvious modifications. A similar procedure can be exploited to deal with non-homogeneous boundary conditions, see \cite{Servadei-Valdinoci13} for some examples of this transition from an abstract result to concrete applications.
}\fr\end{remark}

\subsection{The case of metric measure spaces}
We shall now discuss the case of metric measure structures and how to use the Lewy-Stampacchia inequality to build functions with bounded Laplacian on spaces with a lower bound on the Ricci curvature. In the next section we are going to recall those concepts and results that we shall need without giving full details about relevant definitions. This choice is made to keep the presentation short; we refer to the bibliographical references for all the necessary details. 
\subsubsection{Preliminary notions}

For the purpose of the discussion here, a {\bf metric measure space} is a triple $(X,\sfd,\mm)$ where $(X,\sfd)$ is a complete and separable metric space and $\mm$ is a non-negative Radon measure on it which gives positive mass to every open set.

Given such a m.m.\ space, there is an established definition of the {\bf Sobolev space} $W^{1,2}(X,\sfd,\mm)$, see for instance  \cite{Heinonen07} and \cite{AmbrosioGigliSavare11-3} and references therein. To any $f\in W^{1,2}(X,\sfd,\mm)$ it is associated a function $|Df|\in L^2(X,\mm)$ called minimal weak upper gradient which reduces to  the modulus of the distributional differential when the base space is the Euclidean one. Among others, a basic property of minimal weak upper gradients is their locality, i.e.\ for every $f,g\in W^{1,2}(X,\sfd,\mm)$ we have
\begin{equation}
\label{eq:localgrad}
|Df|=|Dg|,\qquad\mm\text{-a.e.\ on}\ \{f=g\}.
\end{equation}
On proper spaces, one can use this property to define the space $W^{1,2}_{\rm loc}(\Omega)$ for $\Omega\subset X$ open as the subset of $L^2_{\rm loc}(\Omega)$ made of functions $f$ such that $\nchi f\in W^{1,2}(X,\sfd,\mm)$ for every $\nchi\in {\rm Lip}_c(\Omega)$. For $f\in W^{1,2}_{\rm loc}(\Omega)$ the map $|Df|\in L^2_{\rm loc}(\Omega)$ is then defined by
\[
|Df|:=|D(\nchi f)|,\qquad\mm\text{-a.e.\ on}\ \{\nchi =1\},
\]
and the space $W^{1,2}(\Omega)$ is the space of $f\in W^{1,2}_{\rm loc}(\Omega)\cap L^2(\Omega)$ such that $|Df|\in L^2(\Omega)$.

The space $W^{1,2}(X,\sfd,\mm)$ is a Banach space w.r.t.\ the norm $\|f\|_{W^{1,2}}^2:=\|f\|_{L^2}^2+\||Df|\|_{L^2}^2$ and
the energy functional $E:W^{1,2}(X,\sfd,\mm)\to[0,\infty)$ is given by
\[
E(f):=\frac12\int_X|Df|^2\,\d\mm.
\]
We say that $(X,\sfd,\mm)$ is {\bf infinitesimally strictly convex} provided $E:W^{1,2}(X,\sfd,\mm)\to\R$ is differentiable, or equivalently provided for every $f,g\in W^{1,2}(X,\sfd,\mm)$ the limit
\[
\lim_{\eps\to 0}\frac{|D(g+\eps f)|^2-|Dg|^2}{2\eps}
\]
exists in $L^1(X,\mm)$ and {\bf infinitesimally Hilbertian} provided
\[
E(f+g)+E(f-g)=2E(f)+2E(g),\qquad\forall f,g\in W^{1,2}(X,\sfd,\mm),
\]
or equivalently if $W^{1,2}(X,\sfd,\mm)$ is an Hilbert space (see \cite{Gigli12}). It is easy to see that infinitesimally Hilbertian spaces are infinitesimally strictly convex.

Given $\Omega\subset X$ open, the space $W^{1,2}_0(\Omega)$ is the closed subspace of $W^{1,2}(X,\sfd,\mm)$ made of functions which are $\mm$-a.e. 0 outside $\Omega$. Clearly it is canonically and continuously embedded in $W^{1,2}(\Omega)$ and is a lattice w.r.t.\  the usual a.e.\ ordering. We denote as usual by $W^{-1,2}(\Omega)$ its topological dual and by $W^{-1,2}_\lleq(\Omega)$ its topological order dual. As in the Euclidean case, functionals  in $W^{-1,2}_\lleq(\Omega)$ can be represented as measures: the discussion is the same we did in point $(ii)$  after Theorem \ref{LSclassic}, the only difference is that at this level of generality it is not known whether ${\rm Lip}_c(\Omega)$ is dense in $W^{1,2}_0(\Omega)$, which a priori might raise troubles to prove the faithfulness of the representation. Yet, the same argument can be carried out noticing that
\begin{equation}
\label{eq:densevariant}
\text{a positive continuous functional on $W^{1,2}_0(\Omega)$ which is 0 on ${\rm Lip}_c(\Omega)$ is identically 0.}
\end{equation}
Indeed, for every bounded $f\in W^{1,2}_0(\Omega)$ with compact support (= $f$ is 0 $\mm$-a.e.\ outside a certain compact) there are $g_1,g_2\in {\rm Lip}_c(\Omega)$ such that $g_1\leq f\leq g_2$ $\mm$-a.e., so that any functional as in \eqref{eq:densevariant} must be 0 on $f$. Then \eqref{eq:densevariant} follows noticing that, by  standard truncation and (Lipschitz) cut-off arguments,  the subspace  of $W^{1,2}_0(\Omega)$ made of functions bounded and with compact support is strongly dense in $W^{1,2}_0(\Omega)$.

Given $f\in W^{1,2}(X,\sfd,\mm)$ and $\Omega\subset X$ open, the map $E_{f,\Omega}:W^{1,2}_0(\Omega)\to \R$ given by
\[
E_{f,\Omega}(g):=\frac12\int_\Omega |D(f+g)|^2\,\d\mm,
\]
is convex and continuous and, on infinitesimally strictly convex spaces, it is also differentiable. In this latter case we say that $f\in W^{1,2}(X,\sfd,\mm)$ has {\bf measure valued distributional Laplacian} in $\Omega$ provided the only element in $\partial E_{f,\Omega}(0)\subset W^{-1,2}(\Omega)$ belongs to $W^{1,2}_\lleq(\Omega)$ and in this case we write $f\in D(\bd,\Omega)$. The discussion made before shows that this definition is equivalent to the one proposed in \cite{Gigli12} and we shall denote the measure representing $-\partial E_{f,\Omega}(0)$ as $\bd f\restr\Omega$, or simply $\bd f$ in case $\Omega=X$.

\bigskip

For the definition of the {\bf Curvature-Dimension condition} $\CD^*(K,N)$ we refer to \cite{BacherSturm10} (see also \cite{Sturm06II} for the `original' $\CD(K,N)$ condition). One of the main results obtained in  \cite{Gigli12} (see also \cite{Gigli-Mosconi14} for a simplified proof in the infinitesimally Hilbertian case) is the Laplacian comparison estimate for the squared distance on $\CD^*(K,N)$ spaces. For the purposes of the discussion here, it is sufficient to recall it in the following suboptimal form. Recall that for given $\psi:X\to\R\cup\{\pm\infty\}$ and $t>0$, the function $Q_t\psi
:X\to\R\cup\{-\infty\}$ is defined as
\[
Q_t\psi(x):=\inf_{y\in X}\frac{\sfd^2(x,y)}{2t}+\psi(y),
\]
that the $c$-transform $\psi^c$ is defined as $\psi^c:=Q_1(-\psi)$, that $\varphi:X\to\R\cup\{-\infty\}$ is said $c$-concave provided it is not identically $-\infty$ and $\varphi=\psi^c$ for some $\psi:X\to\R\cup\{-\infty\}$.
\begin{theorem}[Laplacian comparison estimates] \label{laplcomp} For given $K\in\R$ and $N\in[1,\infty)$  there is a continuous function $\ell_{K,N}:[0,\infty)\to[0,\infty)$ such that the following holds. 

Let  $(X,\sfd,\mm)$ be an infinitesimally strictly convex $\CD^*(K,N)$ space. Then for every $c$-concave function $\varphi\in {\rm Lip}(X)\cap W^{1,2}(X,\sfd,\mm)$  we have $\varphi\in D(\bd,X)$ with 
\[
\bd\varphi\leq \ell_{K,N}({\rm Lip}(\varphi))\,\mm.
\]
\end{theorem}
This result and the Lewy-Stampacchia inequality are sufficient to build cut-off functions with compact support and bounded Laplacian in infinitesimally strictly convex $\CD^*(K,N)$ spaces. We shall also recall the following fact about evolution of Kantorovich potentials along a $W_2$-geodesic in metric spaces, referring to  Theorem 7.36 in \cite{Villani09} or Theorem 2.18 in \cite{AmbrosioGigli11} for a proof. Recall that given $\mu,\nu\in\probt X$, a function $\varphi:X\to\R\cup\{-\infty\}$ is said Kantorovich potential from $\mu$ to $\nu$ provided it is $c$-concave and a maximizer for the dual problem of optimal transport.
\begin{proposition}[Evolution of Kantorovich potentials]\label{prop:kantev}
Let $(X,\sfd)$ be a metric space, $(\mu_t)\subset \probt X$ a $W_2$-geodesic and $\varphi: X\to\R\cup\{-\infty\}$ a Kantorovich potential from $\mu_0$ to $\mu_1$.

Then for every $t\in [0,1]$:
\begin{itemize} 
\item the function $tQ_t(-\varphi)$ is a Kantorovich potential from $\mu_t$ to $\mu_0$,
\item  the function $(1-t)Q_{1-t}(-\varphi^c)$ is a Kantorovich potential from $\mu_t$ to $\mu_1$.
\end{itemize}
Furthermore, for every $t\in[0,1]$ it holds
\begin{equation}
\label{eq:intpot}
\begin{split}
Q_t(-\varphi)+Q_{1-t}(-\varphi^c)&\geq 0,\qquad\textrm{ everywhere},\\
Q_t(-\varphi)+Q_{1-t}(-\varphi^c)&=0,\qquad\textrm{on } \supp(\mu_t).
\end{split}
\end{equation}
\end{proposition}
This proposition, coupled with the Lewy-Stampacchia inequality and the Laplacian comparison estimate, allows to produce a sort of regularized Kantorovich potentials from intermediate times along a $W_2$-geodesic in infinitesimally strictly convex $\CD^*(K,N)$ spaces, see Theorem \ref{thm:regkant} and the discussion after it for precise statements.

In general $\CD^*(K,N)$ spaces we don't know whether cut-off functions with bounded Laplacian and regularized Kantorovich potentials can be built Lipschitz. In order to get this further property we need to work on infinitesimally Hilbertian   $\CD^*(K,N)$ spaces, also called  $\RCD^*(K,N)$ spaces (\cite{AmbrosioGigliSavare11-2}, \cite{Gigli12}, \cite{AmbrosioGigliMondinoRajala12}, \cite{Erbar-Kuwada-Sturm13}). This enhanced regularity  can be obtained either as a consequence of the general results established in \cite{Jiang12} and \cite{Kell13} concerning Lipschitz continuity of functions with bounded Laplacian, or, as we will do, from the Lipschitz continuity of harmonic functions on $\RCD^*(K,N)$ spaces obtained in \cite{Jiang13} (see also \cite{Kell13} and \cite{Hua-Kell-Xia13} and references therein) and known techniques in the study of the obstacle problem. The advantage of choosing this second approach is that we will obtain  Lipschitz continuity of the solution of the obstacle problem independently of the Laplacian comparison estimates.

In order to pursue this plan we need to recall some results about non-linear potential theory in metric measure spaces. Key facts are that $\CD^*(K,N)$, $N<\infty$, spaces are {\bf doubling} (\cite{Sturm06II}) and support a {\bf weak local 1-2 Poincar\'e inequality} (\cite{Lott-Villani07}, \cite{Rajala12}) and a number of consequences about the behavior of harmonic functions can be deduced from these informations, see \cite{Bjorn-Bjorn11} for an overview on the subject. We shall recall those results we need without aiming at any generality, but only focussing in the content relevant for our discussion.

We start noticing that for $f\in W^{1,2}(\Omega)$, $\Omega$ being an open subset of an infinitesimally strictly convex $\CD^*(K,N)$ space, to be in $D(\bd,\Omega)$ with $\bd f\restr\Omega\leq 0$ is the same as to have the minimization property
\begin{equation}
\label{eq:permin}
\int_\Omega |Df|^2\,\d\mm\leq\int_\Omega |D(f+g)|^2\,\d\mm,\qquad\forall g\in W^{1,2}_0(\Omega),\ g\geq 0
\end{equation}
see \cite{Gigli12} and \cite{Gigli-Mondino12} for the details. Similarly for non-negative Laplacian and non-positive perturbations. In particular,  we can unambiguously define harmonic functions either as those having 0 Laplacian or as local minimizers of the energy.

We then have the following results:
\begin{theorem}[Basic facts about harmonic functions]\label{harmonic}
Let $(X,\sfd,\mm)$ be a $\CD^*(K,N)$ space, $N<\infty$, and $\Omega\subset X$ a bounded open set. Then the following hold.
\begin{itemize}
\item[i)] \underline{Harnack inequality}.  There exist constants $c,\lambda>1$ depending only on $K,N$ and not on $\Omega$  such that the following holds. Let   $f\in W^{1,2}(\Omega)$ be harmonic on  $\Omega$ and non-negative. Then for every $x\in\Omega$ and $r>0$ such that $B_{\lambda r}(x)\subset \Omega$ we have
\begin{equation}
\label{eq:harnack}
\esssup_{B_r(x)}f\leq c\essinf_{B_{\lambda r}(x)}f.
\end{equation}
In particular, harmonic functions have continuous representatives.
\item[ii)]\underline{Strong maximum principle}. Let   $f\in W^{1,2}(\Omega)$ be harmonic on $\Omega$ and assume that its continuous representative has a maximum in a point $x_0\in \Omega$.  Then it is constant on the connected component of $\Omega$ containing $x_0$.

\item[iii)]\underline{Existence and uniqueness of harmonic functions}. Assume that $\mm(X\setminus\Omega)>0$ and let $f\in W^{1,2}(X)$. Then there exists a unique harmonic function $g\in W^{1,2}(\Omega)$ in $\Omega$ such that $f-g\in W^{1,2}_0(\Omega)$. 

\item[iv)]\underline{Comparison principles}. With the same notation and assumptions of the point above, assume furthermore that $\bd f\restr\Omega\leq 0$. Then $f\geq g$ $\mm$-a.e. on $\Omega$. On the other hand, if $f\geq 0$ $\mm$-a.e. on $\Omega$, then $g\geq 0$ $\mm$-a.e. on $\Omega$.
\end{itemize}
\end{theorem}
All these statements are valid in the broader class of doubling spaces supporting a weak-local 1-2 Poincar\'e inequality, see \cite{Bjorn-Bjorn11} for the proofs and detailed bibliography.

Similarly, in the theorem below we collect the basic properties of minima of the obstacle problem that we shall need later on, see  \cite{Farnana09} for the proof.
\begin{theorem}[Basic properties of minima of the obstacle problem]\label{contrepr}
Let $(X,\sfd,\mm)$ be a $\CD^*(K,N)$ space, $\Omega\subset X$ a bounded open set and $\varphi,\psi\in W^{1,2}(\Omega)$ be with $\varphi\leq \psi$ $\mm$-a.e. and $f\in W^{1,2}(X)$. Put 
\[
\mathcal K(\varphi,\psi,f):=\big\{u\in W^{1,2}(\Omega)\ :\ \varphi\leq u\leq\psi\ \mm\text{-a.e. and }u-f\in W^{1,2}_0(\Omega)\big\},
\]
where we wrote for brevity $u-f$ to intend the function defined as $u-f$ in $\Omega$ and as 0 on $X\setminus \Omega$.

Assume that $\mathcal K(\varphi,\psi,f)$ is non-empty.  Then a minimizer  $\bar u$ of $E$ on $[\varphi,\psi]$ exists. Moreover, if $\varphi$ and $\psi$ have continuous representatives then $\bar u$ has a continuous representative as well and if $\mm(X\setminus\Omega)>0$ the minimum is unique.
\end{theorem}

We conclude this introduction recalling the local Lipschitz regularity of harmonic functions on $\RCD^*(K,N)$ spaces. Unlike Theorems \ref{harmonic} and \ref{contrepr} above, here the lower Ricci curvature bound plays a crucial role:

\begin{theorem}[Lipschitz continuity of harmonic functions on $\RCD^*(K,N)$ spaces]\label{lipharm} There exists a constant $\mathcal C=\mathcal C(K,N)$ such that the following holds.
Let $(X,\sfd,\mm)$ be a $\RCD^*(K,N)$ space, $\Omega\subset X$ an open set and $ \bar u\in W^{1,2}(\Omega)$ harmonic.

Then for every $x\in \Omega$ and every $r\in(0,1)$ such that $B_{2r}(x)\subset \Omega$ we have
\begin{equation}
\label{eq:stimalip}
{\rm Lip}(\bar u\restr{B_r(x)})\leq \frac{\mathcal C}r\frac{1}{\mm(B_{2r}(x))}\int_{B_{2r}(x)}|\bar u|\,\d\mm,
\end{equation}
having identified $\bar u$ with its continuous representative.
\end{theorem}
The proof is given in \cite{Jiang13} provided the Dirichlet energy is the natural one in the $\RCD^*(K,N)$ and one uses the calculus tools developed in \cite{Gigli12} (see also \cite{Kell13}, \cite{Hua-Kell-Xia13} and references therein for further details on the topic).

\subsubsection{Lewy-Stampacchia inequality on metric measure spaces}

Given a metric measure space $(X,\sfd,\mm)$ and $\Omega\subset X$ open, the energy functional $E:W^{1,2}_0(\Omega)\to \R$ is clearly convex, continuous and, thanks to the locality property \eqref{eq:localgrad}, submodular. Therefore a direct application of the general Theorem \ref{LS} yields the following regularity result for solutions of the double obstacle problem. Both for simplicity and in view of the forthcoming applications, we state it on infinitesimally strictly convex spaces so that the Laplacian is uniquely defined, but an analogous result holds on every m.m.\ space:
\begin{theorem}[Lewy-Stampacchia inequality on metric measure spaces]
\label{LSmetric}
Let $(X, \sfd, \mm)$ be an infinitesimally strictly convex metric measure space, $\Omega\subset X$ open and   $\varphi, \psi\in W^{1,2}_0(\Omega)\cap D(\bd, \Omega)$ with $\varphi\leq \psi$ $\mm$-a.e..

Then for every minimizer $\bar u$ of   $E$ over $[\varphi, \psi]\subset W^{1,2}_0(\Omega)$ we have $\bar u\in D(\bd,\Omega)$ with
\[
\bd\varphi\restr\Omega\land 0\lleq \bd  \bar u\restr\Omega\lleq  \bd  \psi\restr\Omega\lor 0.
\]
\end{theorem}

\subsubsection{Lipschitz regularity for minima of the obstacle problem on $\RCD^*(K,N)$  spaces}\label{se:lip}
Here we shall prove that on $\RCD^*(K,N)$ spaces, the minimum of the double obstacle problem for given Lipschitz obstacles is Lipschitz itself. This result is independent on the Lewy-Stampacchia inequality: its the proof quite standard once Lipschitz continuity of harmonic functions is known.
\begin{proposition}
\label{lipschitzianity} Let $(X, \sfd, \mm)$ be a $\RCD^*(K, N)$ space, $\Omega\subset X$ open and  bounded and  $\varphi,\psi\in W^{1,2}_0(\Omega)\cap \Lip(\overline \Omega)$. Furthermore, let $\bar u\in W^{1,2}_0(\Omega)$ be a minimizer of $E$ on $[\varphi,\psi]$. 

Then $\bar u$ has a Lipschitz representative, still denoted by $\bar u$, and the bound
\beq
\label{lipu}
\Lip(\bar u)\leq 2\mathcal C\lambda(1+\lambda)(1+c)\Lip(\varphi)\lor \Lip(\psi),
\eeq
holds, where $c,\lambda$ are the constants in the Harnack inequality \eqref{eq:harnack} and $\mathcal C>1$ the one appearing in the Lipschitz estimate \eqref{eq:stimalip}.
\end{proposition}
\begin{proof} Observe that $\varphi$, $\psi$ and $\bar u$ are functions in $W^{1,2}(X)$ vanishing $\mm$-a.e.\ in $X\setminus\Omega$. Applying Theorem \ref{contrepr} on a bounded neighborhood of $\overline\Omega$ with $f\equiv 0$ we deduce that on such neighborhood, and thus on the whole $X$, $\bar u $ has a continuous representative. Denoting still by $\bar u$ this continuous representative, our aim becomes to prove that $\bar u:X\to\R$ is Lipschitz. Recalling that $X$ is geodesic, to conclude it is sufficient to prove that the bound \eqref{lipu} holds for the local Lipschitz constant, i.e. that
\beq
\label{lipu2}
{\rm lip} \, \bar u(x)\leq 2\mathcal C\lambda(1+\lambda)(1+c)\Lip(\varphi)\lor \Lip(\psi),\qquad\forall x\in X,
\eeq
where ${\rm lip} \, \bar u(x):=\lims_{y\to x}\frac{|u(y)-u(x)|}{\sfd(y,x)}$.
Put $C_\varphi:=\{\bar u =\varphi\}$, $C_\psi:=\{\bar u=\psi\}$, $C:=C_\varphi\cap C_\psi$ and define the function $\delta: C_\varphi\cup C_\psi\to(0,+\infty]$ as
\[
\delta(x)=\left\{
\begin{array}{ll}
\frac{\sfd(x,C)}\lambda,&\qquad\text{ if }x\notin C,\\
+\infty,&\qquad\text{ if }x\in C,
\end{array}
\right.
\]
and the constant $L:=\lambda (1+c)\Lip(\varphi)\lor \Lip(\psi)$. We claim that
\beq
\label{claim2}
x_0\in C_\varphi\cup C_\psi,\quad x\in B_{\delta(x_0)}(x_0) \qquad \qquad \Rightarrow \qquad\qquad |\bar u(x)-\bar u(x_0)|\leq L\sfd (x, x_0),
\eeq
which in particular yields \eqref{lipu2} for $x_0\in C_\varphi\cup C_\psi$. This is obvious for $x_0\in C$, thus assume  $x_0\in C_\varphi\setminus C$. By definition of $\delta(x_0)$ and the equivalence stated in inequality \eqref{eq:permin} and the discussion preceding it, we deduce $\bar u\in D(\bd,B_{\lambda\delta(x_0)}(x_0))$ with $\bd \bar u\restr{B_{\lambda\delta(x_0)}}\leq 0$. Since $\bar u\geq \varphi$, it holds 
\[
\bar u(x)\geq \varphi(x)\geq \varphi(x_0)- \Lip(\varphi ) \sfd(x, x_0)=\bar u(x_0)- \Lip(\varphi )\sfd(x, x_0)\qquad \forall x\in X,
\] 
hence it suffices to prove
\beq
\label{claim3}
u(x)\leq u(x_0)+L\sfd(x, x_0), \qquad \forall x\in B_{\delta(x_0)}(x_0).
\eeq
Fix $x\in B_{\delta(x_0)}(x_0)$, put $\rho:=\sfd(x,x_0)\leq \delta(x_0)$,  $v:=u-u(x_0)+\lambda \Lip(\varphi )\rho$ and notice that $v\in D(\bd,B_{\lambda\rho}(x_0))$ with $v\restr{B_{\lambda\rho}(x_0)}\geq0$ and $\bd v\restr{B_{\lambda\rho}(x_0)}\leq0$. Let $v_1$ be the harmonic function on $B_{\lambda\rho}(x_0)$ with the same boundary data as $v$ (point $(ii)$ of Theorem \ref{harmonic}) and put $v_2:=v-v_1$. By point $(iii)$ of Theorem \ref{harmonic} we see that $v\geq v_1\geq 0$ on $B_{\lambda\rho}(x_0)$. Taking into account the Harnack inequality \eqref{eq:harnack} we deduce
\[
v_1(x)\leq\sup_{B_{\delta(x_0)}(x_0)}v_1\leq c v_1(x_0)\leq cv(x_0)= c\lambda \Lip(\varphi )\rho.
\]
Moreover $v\geq v_2\geq 0$ and $v_2\in W^{1,2}_0(B_{\lambda\rho}(x_0))\cap D(\bd,B_{\lambda\rho}(x_0))$ with $\bd v_2\restr{B_{\lambda\rho}(x_0)}\leq0$, therefore by the maximum principle in point $(ii)$ of Theorem \ref{harmonic}, it attains its maximum at some $\bar x\in \supp(\bd v_2\restr{B_{\lambda\rho}(x_0)})$ (because $v_2$ is harmonic on $B_{\lambda\rho}(x_0)\setminus\supp(\bd v_2\restr{B_{\lambda\rho}(x_0)})$). Since $\bd v_2\restr{B_{\lambda\rho}(x_0)}=\bd v\restr{B_{\lambda\rho}(x_0)}=\bd \bar u\restr{B_{\lambda\rho}(x_0)}$ and clearly ${\rm supp}(\bd u\restr{B_{\lambda\rho}(x_0)})\subseteq C_\varphi\cap \bar B_{\lambda\rho}(x_0)$, it holds
\[
\begin{split}
v_2(x)\leq \sup_{B_{\lambda\rho}(x_0)}v_2= v_2(\bar x)\leq v(\bar x)&=u(\bar x)-u(x_0)+\lambda  \Lip(\varphi )\rho\\
&=\varphi(\bar x)-\varphi(x_0)+\lambda  \Lip(\varphi )\rho\leq 2\lambda \Lip(\varphi )\rho.
\end{split}
\]
The last two inequalities yield $v(x)\leq (c+2)\lambda \Lip(\varphi )\rho$, i.e. $u(x)\leq u(x_0)+ (1+c)\lambda \Lip(\varphi )\rho$, which proves  \eqref{claim3}, and hence \eqref{claim2}, for $x\in C_\varphi\setminus C$. The proof for $x_0\in C_\psi\setminus C$ is entirely analogous. 

It remains to prove \eqref{lipu2} for $x_0\in  U:=X\setminus(C_\varphi\cup C_\psi)$ and to this aim we shall use the bound \eqref{claim2} just proved and the Lipschitz estimate \eqref{eq:stimalip}. Fix $x_0\in U$, let $r:=\sfd(x_0,C_\varphi\cup C_\psi)>0$ and find $x_1\in C_\varphi\cup C_\psi$ such that $\sfd(x_0,x_1)=r$. Two cases may occur: either $2r\leq \delta(x_1)$ or $2r>\delta(x_1)$. 

In the first case, from \eqref{claim2}  and $B_r(x_0)\subset B_{2r}(x_1)$ we deduce that 
\[
 |u(x)-u(x_1)|\leq L\sfd(x,x_1)\leq 2Lr,\qquad \forall x\in B_r(x_0),
 \]
 and hence \eqref{eq:stimalip} applied to the harmonic function $u-u(x_1)$ yields \eqref{lipu2}. 
 
 In the second case find $x_2\in C$ such that $\sfd(x_1,C)=\sfd(x_1,x_2)$, recall the definition of $\delta(x_1)$ to notice that $B_r(x_0)\subset B_{2r(1+\lambda)}(x_2)$, so that
 \[
 |u(x)-u(x_2)|\leq\sfd(x,x_2) \Lip(\varphi)\lor \Lip(\psi)\leq 2r(1+\lambda) \Lip(\varphi)\lor \Lip(\psi),\qquad \forall x\in B_r(x_0),
 \]
 and hence \eqref{eq:stimalip} applied to the harmonic function $u-u(x_2)$ yields \eqref{lipu2}. 
\end{proof} 
\subsubsection{Two constructions on $\CD^*(K,N)$ spaces}

We now turn to the two announced constructions on $\CD^*(K,N)$ spaces:  cut-off functions and regularization of Kantorovich potentials along a geodesic.
\begin{theorem}[Cut-off functions]\label{thm:cutoff}
Let $(X,\sfd,\mm)$ be an infinitesimally strictly convex $\CD^*(K,N)$ space, and $C\subset \Omega\subset X$ with $C$ compact and $\Omega$ open. Then there exists a continuous function $\omega\in W^{1,2}(X,\sfd,\mm)$ identically 1 on $C$, identically 0 on $X\setminus \Omega$ such that $\omega\in D(\bd,X)$ with $\bd\omega\ll\mm$ with bounded density.

If $(X,\sfd,\mm)$ is also infinitesimally Hilbertian (i.e.\ a $\RCD^*(K,N)$ space), then $\omega$ can be chosen to be Lipschitz
\end{theorem}
\begin{proof} Without loss of generality we shall assume that   $\Omega$ is bounded and $\mm (X\setminus\Omega)>0$. Let $r>0$ be given by $r^2:=\inf_{x\in C}\sfd^2(x,X\setminus\Omega)/2$ and define $\varphi,\psi:X\to\R$ as
\[
\begin{split}
\varphi(x):=1-1\land \inf_{y\in C}\frac{\sfd^2(x,y)}{2r^2},\qquad\qquad\psi(x):=1\land\inf_{y\in X\setminus \Omega}\frac{\sfd^2(x,y)}{2r^2}.
\end{split}
\]
By construction, $\varphi$ and $\psi$ are Lipschitz with $\Lip(\varphi)$, $\Lip(\psi)\leq 1/r$, they belong to  $ W^{1,2}_0(\Omega)$ and satisfy  $\varphi\leq \psi$, $\varphi=\psi=0$ in $X\setminus\Omega$ and $\varphi=\psi=1$ in $C$.  Moreover, 
\[
-r^2\varphi(x)=\inf_{y\in X}\frac{\sfd^2(x, y)}{2}+r^2\nchi_C(y),\qquad
r^2\psi(x)=\inf_{y\in X}\frac{\sfd^2(x, y)}{2}+r^2\nchi_\Omega(y).
\]
so that the functions $-r^2\varphi,r^2\psi$ satisfy the assumption of Theorem \ref{laplcomp}. Therefore the 1-homogeneity of the Laplacian (which is a direct consequence of the definition) gives
\[
\bd\varphi\geq -\frac1{r^2}\,\ell_{K,N}(r)\mm, \qquad
\bd\psi\leq\frac1{r^2}\, \ell_{K,N}(r)\, \mm.
\]
The thesis now follows letting $\omega$ be the minimum of $E$ on $[\varphi,\psi]$: existence,  uniqueness  and continuity are granted by Theorem \ref{contrepr} (pick $f\equiv 0$), while the uniqueness of the Laplacian and Theorem \ref{LSmetric} give $\bd\omega\ll\mm$ with $\Big\|\frac{\d\bd\omega}{\d\mm}\Big\|_\infty\leq \ell_{K, N}(r)/r^2$. The second part of the statement then follows from the first and Theorem \ref{lipschitzianity}.
\end{proof}
We now turn to the regularization of Kantorovich potentials. To keep the discussion simple, we shall assume in the next theorem that the given Kantorovich potential $\varphi$ is Lipschitz with bounded support. Such a $\varphi$ exists if, for instance, the $W_2$-geodesic considered is made of measures with compact supports.  
\begin{theorem}[Regularization of Kantorovich potentials along a geodesic]\label{thm:regkant} Let $(X,\sfd,\mm)$ be an infinitesimally strictly convex $\CD^*(K,N)$ space, $(\mu_t)\subset\probt X$ a $W_2$-geodesic and $\varphi:X\to\R$ a Kantorovich potential inducing it. Assume that $\varphi$ is Lipschitz with compact support.

Then for every $t\in(0,1)$ there exists a function $\eta_t\in C_c(X)$ such that 
\begin{subequations}
\label{eq:etat}
\begin{align}
\label{eq:boundeta}
&- Q_t(-\varphi)\leq \eta_t\leq Q_{1-t}(-\varphi^c),\\
\label{eq:etacc}
&(t\eta_t)^{cc}(x)=t\eta_t(x)\qquad\text{and}\qquad (-(1-t)\eta_t)^{cc}(x)=-(1-t)\eta_t(x),\qquad\forall x\in\supp\mu_t,
\end{align}
\end{subequations}
belonging to $D(\bd, X)$ with $\bd\eta_t\ll\mm$ and 
\begin{equation}
\label{eq:etalap}
\Big\|\frac{\d\bd\eta_t}{\d\mm}\Big\|_{L^\infty}\leq \frac{\ell_{K,N}(2\sqrt{t\|\varphi\|_{L^\infty}})}{t}\lor\frac{\ell_{K,N}(\sqrt{2(1-t)\|\varphi\|_{L^\infty}})}{1-t}
\end{equation}
If $(X,\sfd,\mm)$ is also infinitesimally Hilbertian (i.e.\ a $\RCD^*(K,N)$ space), then $\eta_t$ can be chosen to be Lipschitz.
\end{theorem}
\begin{proof} Directly from the definition we see that if $\sfd^2(x,\supp\varphi)\geq 2\|\varphi\|_{L^\infty}$, then $Q_t(-\varphi)(x)=0$ for every $t\in(0,1)$. It follows that $\supp(Q_t(-\varphi))$ and $\supp(Q_t(-\varphi^c))$ are uniformly bounded for $t\in(0,1]$. Recalling that $\CD^*(K,N)$ spaces are proper (because they are doubling), we conclude that the sets $\supp(Q_t(-\varphi))$ and $\supp(Q_t(-\varphi^c))$ are all contained in some fixed compact set $C$ for all $t\in(0,1)$. By definition, it is also easy to check that the minimum in the definition of $Q_t(-\varphi)(x)$ is reached at a point $x_t$ such that $\sfd^2(x,x_t)\leq 4t\|\varphi\|_{L^\infty}$. It follows that $\Lip(Q_t(-\varphi))\leq 2\sqrt{\|\varphi\|_{L^\infty}/t}$ and similarly $\Lip(Q_{1-t}(-\varphi^c))\leq2 \sqrt{\|\varphi\|_{L^\infty}/(1-t)}$.
Clearly, $-Q_t(-\varphi), Q_{1-t}(-\varphi^c)\in W^{1,2}_0(X)$ and, by the first in \eqref{eq:intpot} we know that $- Q_t(-\varphi)\leq Q_{1-t}(-\varphi^c)$. 

We apply Theorem \ref{contrepr} on a bounded neighborhood $\Omega$ of $C$ with  $f\equiv0$: we deduce the existence of a minimum  $\eta_t\in W^{1,2}_0(\Omega)$ of $E$  on $[- Q_t(-\varphi), Q_{1-t}(-\varphi^c)]\subset W^{1,2}_0(\Omega)$ which has a continuous representative, still denoted by $\eta_t$. Moreover $\supp\eta_t\subseteq C$, so that $\eta_t\in C_c(X)$. To check  \eqref{eq:etacc}, notice that directly from the definition one has that for arbitrary $\psi:X\to\R$, the function $\psi^{cc}$ is the least $c$-concave function greater or equal than $\psi$ everywhere on $X$, so that the claim follows from the $c$-concavity of $t Q_t(-\varphi)$ and $(1-t)Q_{1-t}(-\varphi^c)$ and the second in \eqref{eq:intpot}, which together with \eqref{eq:boundeta} gives
\[
-Q_t(-\varphi)= \eta_t= Q_{1-t}(-\varphi^c)\qquad\text{on }\supp\mu_t.
\]
For \eqref{eq:etalap} notice that  $t Q_t(-\varphi)$ is $c$-concave with $\Lip(t Q_t(-\varphi))\leq2 \sqrt{t\|\varphi\|_{L^\infty}}$, so that by Theorem \ref{laplcomp} we deduce $\bd(t Q_t(-\varphi))\leq \ell_{K,N}( 2\sqrt{t\|\varphi\|_{L^\infty}})$. Similarly $\bd Q_{1-t}(-\varphi^c)\leq \ell_{K,N}(2 \sqrt{(1-t)\|\varphi\|_{L^\infty}})$, so that the 1-homogeneity of the Laplacian and Theorem \ref{LSmetric} yield \eqref{eq:etalap}.

The last statement concerning Lipschitz regularity is then a consequence of the construction and Theorem \ref{lipschitzianity}.
\end{proof}
We remark that in general the function $t\eta_t$ (resp. $-(1-t)\eta_t$) produced by the previous theorem is not $c$-concave, yet \eqref{eq:etacc} grants that, in a sense, it is `$c$-concave in the region of interest', i.e.:
\[
\begin{split}
t\eta_t(x)+(t\eta_t)^c(y)&\leq\frac{\sfd^2(x,y)}{2},\qquad\forall (x,y)\in X^2,\\
t\eta_t(x)+(t\eta_t)^c(y)&=\frac{\sfd^2(x,y)}{2},\qquad\forall (x,y)\in\supp\ggamma,
\end{split}
\]
for every optimal plan $\ggamma$ from $\mu_t$ to $\mu_0$ (resp.\ from $\mu_t$ to $\mu_1$).

\def\cprime{$'$} \def\cprime{$'$}

\end{document}